\documentclass[11pt]{article}
\usepackage[numbers,sort&compress]{natbib}
\usepackage{hypernat} 
\usepackage{amsmath,amsthm,amsfonts,amssymb,enumerate,graphicx,calc,rsfso}
\usepackage[tmargin=25mm,bmargin=25mm,rmargin=30mm,lmargin=30mm]{geometry}
\usepackage[colorlinks=true,citecolor=black,linkcolor=black,urlcolor=blue,pagebackref]{hyperref}

\renewcommand*{\backref}[1]{}
\renewcommand*{\backrefalt}[4]{\footnotesize%
    \ifcase #1 \vspace*{0cm}\hfill{\mbox{[not cited]}}%
    \or         \vspace*{0cm}\hfill{\mbox{[p.\,#2]}}%
    \else     \vspace*{0cm}\hfill{\mbox{[pp.\,#2]}}%
    \fi}

\usepackage[noabbrev,capitalise]{cleveref}
\crefname{lem}{Lemma}{Lemmas}
\crefname{thm}{Theorem}{Theorems}
\crefname{prop}{Proposition}{Propositions}

\theoremstyle{plain}
\newtheorem{theorem}{Theorem}
\newtheorem{lemma}[theorem]{Lemma}

\newtheorem{conj}[theorem]{Conjecture}

\theoremstyle{definition}

\renewcommand{\baselinestretch}{1.15}

\renewcommand{\thefootnote}{\fnsymbol{footnote}}	

\makeatletter
\renewcommand\section{\@startsection {section}{1}{\z@}%
                                   {-3ex \@plus -1ex \@minus -.2ex}%
                                   {2ex \@plus.2ex}%
                                   {\normalfont\large\bfseries}}
\renewcommand\subsection{\@startsection{subsection}{2}{\z@}%
                                     {-2.5ex\@plus -1ex \@minus -.2ex}%
                                     {1.5ex \@plus .2ex}%
                                     {\normalfont\normalsize\bfseries}}
\renewcommand\subsubsection{\@startsection{subsubsection}{3}{\z@}%
                                     {-2ex\@plus -1ex \@minus -.2ex}%
                                     {1ex \@plus .2ex}%
                                     {\normalfont\normalsize\bfseries}}
 \renewcommand\paragraph{\@startsection{paragraph}{4}{\z@}%
                                    {1.5ex \@plus.5ex \@minus.2ex}%
                                    {-1em}%
                                    {\normalfont\normalsize\bfseries}}
\renewcommand\subparagraph{\@startsection{subparagraph}{5}{\parindent}%
                                       {1.5ex \@plus.5ex \@minus .2ex}%
                                       {-1em}%
                                      {\normalfont\normalsize\bfseries}}
\makeatother

\newcommand{\arXiv}[1]{arXiv:\,\href{http://arxiv.org/abs/#1}{#1}}

\newcommand{\msn}[1]{MR:\,\href{http://www.ams.org/mathscinet-getitem?mr=MR#1}{#1}}

\newcommand{\doi}[1]{doi:\,\href{http://dx.doi.org/#1}{#1}}

\newcommand{\Oh}[1]{\ensuremath{\protect O(#1)}}
\newcommand{\half}{\ensuremath{\protect\tfrac{1}{2}}}
\renewcommand{\geq}{\geqslant}
\renewcommand{\leq}{\leqslant}
\newcommand{\LEQ}{\ensuremath{\,\leq\,}}

\newcommand{\V}[1]{\ensuremath{\protect|#1|}} 
\newcommand{\E}[1]{\ensuremath{\protect\Vert #1\Vert}} 
\newcommand{\KK}{\ensuremath{K}}
\newcommand{\g}[1]{\ensuremath{\protect G_{#1} }}
\newcommand{\gp}[1]{\ensuremath{\protect G^-_{#1}}}
\newcommand{\gi}{\ensuremath{\protect G_i}}
\newcommand{\gip}{\ensuremath{\protect G^-_i}}
\newcommand{\gj}{\ensuremath{\protect G_j}}

\newcommand{\CR}[1]{\ensuremath{\protect\textsf{\textup{cr}}(#1)}}
\newcommand{\RCR}[1]{\ensuremath{\protect\overline{\textsf{\textup{cr}}}(#1)}}
\newcommand{\CCR}[1]{\ensuremath{\protect\textsf{\textup{cr}}^{\star}(#1)}}

\newcommand{\FFF}{\ensuremath{\mathcal{F}}}
\newcommand{\PPP}{\ensuremath{\mathcal{P}}}

\newcommand\graphminus{-}  

\title{\bf Tight Upper Bounds on the Crossing Number\\ in a Minor-Closed Class\footnote{A preliminary version of this paper was published as ``Improved upper bounds on the crossing number'' in \emph{Proc.\ of  24th Annual Symp.\ on Computational Geometry} (SoCG Õ08), pp. 375--384, ACM, 2008. }}

\author{
Vida Dujmovi\'c\footnote{School of Computer Science and Electrical Engineering, University of Ottawa, 
Ottawa, Canada (\texttt{vida.dujmovic@uottawa.ca}). Supported by NSERC and Ontario Ministry of Research and Innovation.}
\quad 
Ken-ichi Kawarabayashi\footnote{National Institute of Informatics, Tokyo, Japan  (\texttt{k\_keniti@nii.ac.jp}). Research  supported by JSPS KAKENHI Grant Number JP18H05291.}
\\
Bojan Mohar\footnote{Department of Mathematics, Simon Fraser University, Burnaby, Canada; and Department of Mathematics, University of Ljubljana, Ljubljana, Slovenia  (\texttt{mohar@sfu.ca}). Research supported by an NSERC Discovery Grant, CRC program, and in part by ARRS, Research Program P1-0297. }
\quad
David R.~Wood
\footnote{School of Mathematical Sciences, Monash University, Melbourne, Australia (\texttt{david.wood@monash.edu}). Supported by the Australian Research Council.}
}

\begin{document}
\maketitle
\renewcommand{\thefootnote}{\arabic{footnote}}

\begin{abstract}
The crossing number of a graph is the minimum number of crossings in a drawing of the graph in the plane. Our main result is that every graph $G$ that does not contain a fixed graph as a minor has crossing number $\Oh{\Delta n}$, where $G$ has $n$ vertices and maximum degree $\Delta$. This dependence on $n$ and $\Delta$ is best possible. This result answers an open question of Wood and Telle [\emph{New York J.~Mathematics}, 2007], who proved the best previous bound of \Oh{\Delta^2 n}.
We also study the convex and rectilinear crossing numbers, and prove an \Oh{\Delta n} bound for the convex crossing number of bounded pathwidth graphs, and a $\sum_v\deg(v)^2$ bound for the rectilinear crossing number of  $K_{3,3}$-minor-free graphs.
\end{abstract}





\newpage
\pagestyle{plain}
\section{Introduction}
\label{Intro}

The \emph{crossing number} of a graph\footnote{We consider graphs $G$ that are undirected, simple, and finite. Let $V(G)$ and $E(G)$ respectively be the vertex and edge sets of $G$. Let $\V{G}:=|V(G)|$ and $\E{G}:=|E(G)|$. For each vertex $v$ of $G$, let $N_G(v):=\{w\in V(G):vw\in E(G)\}$ be the neighbourhood of $v$ in $G$. The \emph{degree} of $v$, denoted by $\deg_G(v)$, is $|N_G(v)|$. When the graph is clear from the context, we write $\deg(v)$. Let $\Delta(G)$ be the maximum degree of $G$.} $G$, denoted by \CR{G}, is the minimum number of crossings in a drawing\footnote{A \emph{drawing} of a graph represents each vertex by a distinct point in the plane, and represents each edge by a simple closed curve between its endpoints, such that the only vertices an edge intersects are its own endpoints, and no three edges intersect at a common point (except at a common endpoint). A drawing is \emph{rectilinear} if each edge is a line-segment, and is \emph{convex} if, in addition, the vertices are in convex position. A \emph{crossing} is a point of intersection between two edges (other than a common endpoint). A drawing with no crossings is \emph{crossing-free}. A graph is \emph{planar} if it has a crossing-free drawing.} of $G$ in the plane; see \cite{PachToth-JCTB00,Schaefer14} for surveys. The crossing number is an important measure of non-planarity of a graph \cite{Szekely-DM04}, with applications in discrete and computational geometry \cite{PachSharir-CPC98, Szekely-CPC97}, VLSI circuit design \cite{BL84, Leighton83, Leighton84}, and in several other areas of mathematics and theoretical computer science; see \cite{Szekely-DM04} for details. In information visualisation, one of the most important measures of the quality of a graph drawing is the number of crossings \cite{Purchase-JVLC98, PCJ97}.

Computing the crossing number is $\mathcal{NP}$-hard \cite{GJ-SJDM83}, and remains so for simple cubic graphs \cite{Hliney-JCTB06, PSS11}. Moreover, the exact or even asymptotic crossing number is not known for specific graph families, such as complete graphs \cite{RT-AMM97}, complete bipartite graphs \cite{Nahas-EJC03, RS-JGT96, RT-AMM97}, and cartesian products \cite{AR-JCTB04, Bokal-JCTB07, GS-JGT04, RT-DCG95}. On the other hand, \citet{Grohe04} developed a quadratic-time algorithm that decides whether a given graph has crossing number at most some fixed number $k$, and if this is the case, produces a drawing of the graph with at most $k$ crossings. \citet{KR07} improved the time complexity to linear. 

Given that the crossing number seems so difficult, it is natural to focus on asymptotic bounds rather than exact values. The `crossing lemma', conjectured by \citet{EG-AMM73} and first proved\footnote{A remarkably simple probabilistic proof of the crossing lemma with $c_\epsilon=\frac{4\epsilon}{(6+\epsilon)^3}$ was found by Chazelle, Sharir and Welzl (see \cite{Proofs3}). See \cite{Montaron-JGT05, PRTT-DCG06} for subsequent improvements.} by \citet{Leighton83} and \citet{Ajtai82}, gives such a lower bound. It states that for $\epsilon>0$ every graph $G$ with average degree greater than $6+\epsilon$ has $$\CR{G}\geq c_\epsilon\frac{\E{G}^3}{\V{G}^2}.$$
Other general lower bound techniques that arose out of the work of \citet{Leighton83, Leighton84} include the bisection/cut/tree width method \cite{DV-JGAA03, PSS-Algo96, SS-CPC94, SSSV-Algo96,DJMNW17} and the embedding method \cite{SSSV-AM96, SS-CPC94}.

Upper bounds on the crossing number of general families of graphs have been less studied, and are the focus of this paper. Obviously $\CR{G}\leq\binom{\E{G}}{2}$ for every graph $G$. A family of graphs has \emph{linear}\footnote{If the crossing number of a graph is linear in the number of edges then it is also linear in the number of vertices. To see this, let $G$ be a graph with $n$ vertices and $m$ edges. Suppose that $\CR{G}\leq cm$. If $m<4n$ then $\CR{G}\leq4cn$ and we are done. Otherwise  $\CR{G}\geq m^3/64n^2$ by the crossing lemma. Thus $m\leq 8\sqrt{c}n$ and $\CR{G} \leq 8c^{3/2}n$.}
 crossing number if $\CR{G}\leq c\, \V{G}$ for some constant $c$ and for every graph $G$ in the family. The following theorem of \citet{PachToth-GD05} and \citet{BPT-IJFCS06} shows that graphs of bounded Euler genus\footnote{Let $\Sigma$ be a surface. An \emph{embedding} of a graph $G$ in $\Sigma$ is a crossing-free drawing of $G$ in $\Sigma$. The \emph{Euler genus} of $\Sigma$ equals $2h$ if $\Sigma$ is the sphere with $h$ handles, and equals $c$ if $\Sigma$ is the sphere with $c$ cross-caps. The \emph{Euler genus} of a graph $G$ is the minimum Euler genus of a surface in which there is an embedding of $G$. In what follows, by a \emph{face} of an embedded graph  we mean the set of vertices on the boundary of the face. See \citep{MoharThom} for a thorough treatment of graph embeddings.} and bounded degree have linear crossing number.


\begin{theorem}[\cite{PachToth-GD05,BPT-IJFCS06}]
\label{BoundedGenus}
For every integer $\gamma\geq0$, there are constants $c$ and $c'$, such that every graph $G$ with Euler genus $\gamma$ has crossing number
 $$\CR{G}\LEQ c \sum_{v\in V(G)}\deg(v)^2 \LEQ c'\Delta(G)\cdot\V{G}.$$
\end{theorem}


In the case of orientable surfaces, Djidjev and Vr{\v{t}}o~\cite{DV12} greatly improved the dependence on $\gamma$ in \cref{BoundedGenus}, by proving that $\CR{G}\LEQ c\gamma\cdot\Delta(G)\cdot\V{G}$. \citet{WT07} proved that bounded-degree graphs that exclude a fixed graph as a minor\footnote{Let $vw$ be an edge of a graph $G$. Let $G'$ be the graph obtained by identifying the vertices $v$ and $w$, deleting loops, and replacing parallel edges by a single edge. Then $G'$ is obtained from $G$ by \emph{contracting} $vw$. A graph $H$ is a \emph{minor} of a graph $G$ if $H$ can be obtained from a subgraph of $G$ by contracting edges. A family of graphs \FFF\ is \emph{minor-closed} if $G\in\FFF$ implies that every minor of $G$ is in \FFF. \FFF\ is \emph{proper} if it is not the family of all graphs. A deep theorem of Robertson and Seymour~\cite{RS-XX} states that every proper minor-closed family can be characterised by a finite family of excluded minors. Every proper minor-closed family is a subset of the $H$-minor-free graphs for some graph $H$. We thus focus on minor-closed families with one excluded minor.} have linear crossing number.

\begin{theorem}[\cite{WT07}]
\label{CrossingMinorFree}
For every graph $H$, there is a constant $c=c(H)$, such that every $H$-minor-free graph $G$ has crossing number 
$$\CR{G}\LEQ c\,\Delta(G)^2\cdot\V{G}.$$
\end{theorem}

\cref{CrossingMinorFree} is stronger than \cref{BoundedGenus} in the sense that graphs of bounded genus exclude a fixed graph as a minor, but there are graphs with a fixed excluded minor and arbitrarily large genus\footnote{Since the genus of a graph is the sum of the genera of its biconnected components, it is easy to construct graphs with unbounded genus and no $K_5$-minor. (Take $n$ copies of $K_{3,3}$ for example.) A more highly connected example is the complete bipartite graph $K_{3,n}$, which has no $K_5$-minor and has unbounded genus \cite{Ringel65}. Seese and Wessel~\cite{SW-JCTB89} constructed a family of graphs, each with no $K_8$-minor and maximum degree $5$, and with unbounded genus.} On the other hand, \cref{BoundedGenus} has better dependence on $\Delta$ than \cref{CrossingMinorFree}. 
%
%
For other  work on minors and crossing number see \cite{BCSV-ENDM,BFM-SJDM06, BFW, GS-JGT01, GRS-EJC04, Hliney-JCTB03, Hliney-JCTB06, Negami-JGT01, PSS11}.

Note that to obtain a linear upper bound on the crossing number, it is necessary to assume both bounded degree and some structural assumption such as an excluded minor (as in \cref{CrossingMinorFree}). For example, $K_{3,n}$ has no $K_5$-minor, yet its crossing number is $\Omega(n^2)$ \cite{RS-JGT96, Nahas-EJC03}. On the other hand, bounded degree does not by itself guarantee linear crossing number. For example, a random cubic graph on $n$ vertices has $\Omega(n)$ bisection width \cite{CE-BAMS89, DDSW-TCS03}, which implies that its crossing number is $\Omega(n^2)$  \cite{DV-JGAA03, Leighton83}.

Pach and T\'{o}th~\cite{PachToth-GD05} proved that the upper bound in \cref{BoundedGenus} is best possible, in the sense that for all $\Delta$ and $n$, there is a toroidal graph with $n$ vertices and maximum degree $\Delta$ whose crossing number is $\Omega(\Delta n)$. In \cref{lower-bounds} we extend this $\Omega(\Delta n)$ lower bound to graphs with no $K_{3,3}$-minor, no $K_5$-minor, and more generally, with no $K_h$-minor. Our main result is to prove a matching upper bound for all graphs excluding a fixed minor. That is, we improve the quadratic dependence on $\Delta(G)$ in \cref{CrossingMinorFree} to linear.

\begin{theorem}\label{main}
For every graph $H$ there is a constant $c=c(H)$, such that every $H$-minor-free graph $G$ has crossing number 
$$\CR{G}\leq c\,\Delta(G)\cdot \V{G}.$$
\end{theorem}




While our upper bound in \cref{main} is optimal in terms of $\Delta(G)$ and $\V{G}$, it remains open whether every graph excluding a fixed minor has crossing number $\Oh{\sum_v\deg(v)^2}$, as is the case for graphs of bounded Euler genus. Note that a $\sum_{v}\deg(v)^2$ upper bound is stronger than a $\Delta(G)\cdot \V{G}$ upper bound; see \cref{Bounds}. 
\citet{WT07} conjectured that every graph excluding a fixed minor has crossing number $\Oh{\sum_v\deg(v)^2}$. In \cref{k5k33}, we establish this conjecture for $K_{5}$-minor-free graphs, and prove the same bound on the rectilinear crossing number\footnote{The \emph{rectilinear crossing number} of a graph $G$, denoted by $\RCR{G}$, is the minimum number of crossings in a rectilinear drawing of $G$. The \emph{convex crossing number}, denoted by $\CCR{G}$, is the minimum number of crossings in a convex drawing of $G$.} of $K_{3,3}$-minor-free graphs. In addition to these results, we establish in \cref{pw} optimal bounds on the convex crossing number of interval graphs, chordal graphs, and bounded pathwidth graphs.

It is worth noting that our proof is constructive, assuming a structural decomposition (\cref{RS-minors}) by Robertson and Seymour~\cite{RS-XVI} is given. Demaine~et~al.~\cite{DHK-FOCS05} gave a polynomial-time algorithm to compute this decomposition. Consequently, our proof can be converted into a  polynomial-time algorithm that, given a graph $G$ excluding a fixed minor, finds a drawing of $G$ with the claimed number of crossings.

\section{Lower Bounds}
\label{lower-bounds}

In this section we describe graphs that provide lower bounds on the crossing number. The constructions are variations on those by \citet{PachToth-GD05}. We include them here to motivate our interest in matching upper bounds in later sections.

\begin{lemma}
\label{lowerbound-k33}
For all positive integers $\Delta$ and $n$, such that  $\Delta \equiv 0 \pmod 4$ and $n\equiv 0 \pmod{5(\Delta/2-1)}$, there is a (chordal) $K_{3,3}$-minor-free graph $G$ with $\V{G}=n$, $\Delta(G)=\Delta$, and
$$\CR{G}=\frac{\Delta n}{40}\Big(1+\frac{2}{\Delta-2}\Big)\,>\, \frac{\Delta n}{40}.$$
\end{lemma}

\begin{proof}
Start with $K_{5}$ as the base graph. For each edge $vw$ of $K_{5}$, add $\Delta/4-1$ new vertices, each adjacent to $v$ and $w$. The resulting graph $G'$ is chordal and $K_{3,3}$-minor-free, $\Delta(G')=\Delta$, and $\V{G'}=5(\Delta/2-1)$. Take $\frac{n}{5(\Delta/2-1)}$ disjoint copies of $G'$ to obtain a $K_{3,3}$-minor-free graph $G$ on $n$ vertices and maximum degree $\Delta$. Thus $\CR{G}=\CR{G'}\frac{n}{5(\Delta/2-1)}$. 
To complete the proof, we use a standard technique to prove that  $\CR{G'}=(\Delta/4)^2$.
Consider a drawing $D'$ of $G'$ with $\CR{G'}$ crossings. For each vertex of degree $2$ in $G'$, contract one of its incident edges to obtain a drawing $D''$ of a multigraph $G''$. This operation does not affect the number of crossings. Thus $\CR{G'}=\CR{G''}$. Consider an edge $e$ between $v$ and $w$ in $G''$ that is crossed by the least number of edges in $D''$. Redraw each remaining edge between $v$ and $w$ so close to $e$ that the parallel edges between $v$ and $w$ appear as one edge in the drawing. Again, this does not increase the number of crossings. Repeat this step for each  pair of vertices of $G''$. The resulting drawing of $G''$ is equivalent to an optimal drawing of $K_{5}$ where each pair of crossing edges is replaced by $(\Delta/4)^2$ pairs of crossing edges in the drawing of $G''$. Since $\CR{K_{5}}=1$, $\CR{G''}=(\Delta/4)^2=\CR{G'}$.
\end{proof}

A similar technique gives the following lemma.

\begin{lemma}\label{lowerbound}
For every set $D=\{2,d_1, \dots, d_p\}$ of positive integers such that $d_i\equiv 0 \pmod 4$ for $i\in\{1,\dots,p\}$, there are infinitely many (chordal) $K_{3,3}$-minor-free graphs $G$ such that the degree set of $G$ is $D$ and $$ \CR{G}> \frac{1}{200}\sum_{v\in V(G)}\deg(v)^2.$$
\end{lemma}

\begin{proof}
For each $d_i\in D\setminus\{2\}$, let $n_i=\frac{5}{2}d_i-5$. 
By  \cref{lowerbound-k33}, there is a (chordal) $K_{3,3}$-minor-free graph $G_i$ with five vertices of degree $d_i$ and $n_i-5$ vertices of degree $2$, such that 
$$\CR{G_i}>\frac{1}{40} d_in_i > \frac{1}{200} ( 5d_i^2+(n_i-5)2^2) = 
\frac{1}{200} \sum_{v\in V(G_i)} \deg_{G_i}(v)^2 .$$
Every graph $G$ created by taking one or more disjoint copies of each of $G_1,\dots, G_p$ is $K_{3,3}$-minor-free with degree set $D$, and $\CR{G}\geq\frac{1}{200}\sum_v\deg(v)^2$.
\end{proof}

The above results generalise to $K_h$-minor-free graphs, for $h\geq 5$.

\begin{lemma}\label{lowerbound-general}
For every integer $h\geq 5$ and every $\Delta$ such that $\Delta\equiv 0\pmod {h-2}$ for $h\geq 6$ and $\Delta\equiv 0\pmod 3$ for $h=5$, there exists infinitely many $K_{h}$-minor-free graphs $G$ with $\Delta(G)=\Delta$ and $$\CR{G}\geq c\,h\,\Delta \cdot \V{G},$$ for some absolute constant $c$. Moreover, $G$ is chordal for $h\geq 6$.
\end{lemma}

\begin{proof}[Sketch]
For $h=5$, use $K_{3,3}$ as the starting graph. For $h\geq 6$, use $K_{h-1}$.
The remaining arguments follow the proof of \cref{lowerbound-k33} and use the fact that $\CR{K_{3,3}}=1$ and $\CR{K_{h-1}}\in \Theta(h^4)$.
\end{proof}

\section{Linear Bounding Functions}
\label{Bounds}

The upper bounds on the crossing number that are proved in this paper are linear (in the number of vertices) for graphs with bounded degree,
but can be quadratic or more for graphs of unbounded degree. A number of functions satisfy these properties. The smallest such function that we consider is:
\begin{equation}
\label{SumDegreeSquared}
\sum_{v\in V(G)}\deg(v)^2.
\end{equation}
Most of the graphs that we consider have a linear number of edges. Thus it is worthwhile to note that the dependence on the maximum degree in \eqref{SumDegreeSquared} is at most linear. In particular, 
\begin{equation}
\label{SumDegreeSquaredDelta}
\sum_{v\in V(G)}\deg(v)^2
\LEQ
\Delta(G)\sum_{v\in V(G)}\deg(v)
\,=\,
2\Delta(G)\cdot \E{G}.
\end{equation}

The next best function is:
\begin{equation}
\label{SumEdgeDegrees}
\sum_{vw\in E(G)}\deg(v)\deg(w).
\end{equation}
Note the following relationship between \eqref{SumDegreeSquared} and \eqref{SumEdgeDegrees}, which is tight for every regular graph.

\begin{lemma}
\label{DegreeInequality}
For every graph $G$ with minimum degree $\delta$ (ignoring isolated vertices), 
$$\delta \sum_{v\in V(G)}\!\!\!\!\deg(v)^2
\leq 
2\sum_{vw\in E(G)}\!\!\!\!\deg(v)\deg(w).$$
\end{lemma}

\begin{proof}
Observe that 
\begin{align*}
\sum_{vw\in E(G)}\!\!\!\!\deg(v)\deg(w)
\,=\,
\half\sum_{v\in V(G)}\sum_{w\in N_G(v)}\!\!\!\!\deg(v)\deg(w)
\,=\,
\half\sum_{v\in V(G)}\!\!\!\!\deg(v)\sum_{w\in N_G(v)}\!\!\!\!\deg(w).
\end{align*}
Since $w$ is at least adjacent to $v$,
\begin{equation*}
\sum_{vw\in E(G)}\!\!\!\!\deg(v)\deg(w)
\,\geq\,
\half\sum_{v\in V(G)}\!\!\!\!\deg(v)\sum_{w\in N_G(v)}\!\!\!\!\!\delta
\,=\,
\half \delta\sum_{v\in V(G)}\!\!\!\!\deg(v)^2.\qedhere
\end{equation*}
\end{proof}

In certain situations we can conclude the bound in 
\eqref{SumDegreeSquaredDelta} by first proving the seemingly weaker bound in 
\eqref{SumEdgeDegrees}.

\begin{lemma}\label{SubdivTrick}
Let $X$ be a class of graphs closed under taking subdivisions.
Suppose that $$\CR{G}\leq c \sum_{vw\in E(G)} \deg(v) \deg(w)$$ for every graph $G\in X$.
Then $$\CR{G}\leq 2c\, \sum_{v\in V(G)}\deg(v)^2$$ for every graph $G \in X$.
\end{lemma}

\begin{proof}
Let $G\in X$. Let $G'$ be the graph obtained from $G$ by subdividing every edge once. By assumption, $G' \in X$ and
\begin{align*}
\CR{G'}
\leq\,& c \sum_{vw \in E(G')} \deg(v) \deg(w)\\
=\,& c \sum_{vw \in E(G)} (2 \deg(v) + 2\deg(w))\\
=\,& 2c \sum_{vw \in E(G)} (\deg(v) + \deg(w))\\
=\,& 2c \sum_{v \in V(G)} \deg(v)^2.
\end{align*}
The result follows since $\CR{G} = \CR{G'}$.
\end{proof}


We can also conclude a \Oh{\Delta(G)\cdot\V{G}} bound from  $\sum_{vw\in E(G)}\deg(v)\deg(w)$ for sparse graphs.

\begin{lemma}\label{arboricity-trick}
Let $G$ be a graph such that every subgraph of $G$ on $n$ vertices has at most $kn$ edges. Then
$$\sum_{vw\in E(G)}\deg(v)\deg(w)\;\leq\; 16k\cdot\Delta(G)\cdot\E{G}\;\leq\;16k^2\cdot\Delta(G)\cdot\V{G}.$$
\end{lemma}

\begin{proof}
For integers $i,j\geq0$, let
\begin{align*}
\Delta_i&:=\Delta(G)/2^i\\
V_i&:=\{v\in V(G):\Delta_{i+1}<\deg(v)\leq\Delta_i\}\\
n_i&:=|V_i|\\
E_{i,j}&:=\{vw\in E(G):v\in V_i,w\in V_j\}\\
e_{i,j}&:=|E_{i,j}|\enspace.
\end{align*}
Let $S_i:=\{j\geq0:n_j\leq n_i\}$. Thus
\begin{align*}
\sum_{vw\in E(G)}\deg(v)\deg(w)
\;&\leq\;
\sum_{i\geq0}\sum_{j\in S_i}\sum_{vw\in E_{i,j}}\deg(v)\deg(w)\\
\;&\leq\;
\sum_{i\geq0}\sum_{j\in S_i}e_{i,j}\Delta_i\Delta_j\\
\;&\leq\;
k\sum_{i\geq0}\sum_{j\in S_i}(n_i+n_j)\Delta_i\Delta_j\\
\;&\leq\;
2k\sum_{i\geq0}n_i\Delta_i\sum_{j\geq0}\Delta_j\enspace.
\end{align*}
Since $\sum_{j\geq0}\Delta_j\,<\,2\cdot\Delta(G)$,
\begin{align*}
\sum_{vw\in E(G)}\deg(v)\deg(w)
\;&<\;
4k\cdot\Delta(G)\sum_{i\geq0}n_i\Delta_i\enspace.
\end{align*}
Observe that
\begin{align*}
2\E{G}
\,=\,\sum_{i\geq0}\sum_{v\in V_i}\deg(v)
\,>\,\sum_{i\geq0}n_i\Delta_{i+1}
\,=\,\half\sum_{i\geq0}n_i\Delta_i\enspace.
\end{align*}
Thus
\begin{align*}
\sum_{vw\in E(G)}\deg(v)\deg(w)
\;&<\;
16k\cdot\Delta(G)\cdot\E{G}\enspace.\qedhere
\end{align*}
\end{proof}


The next best function is:
\begin{equation}
\label{SumDegreeCubed}
\sum_{v\in V(G)}\deg(v)^3.
\end{equation}
Note the following relationship between \eqref{SumEdgeDegrees} and \eqref{SumDegreeCubed}, which is also tight for regular graphs.

\begin{lemma}
\label{SecondDegreeInequality}
For every graph $G$,
$$\sum_{vw\in E(G)}\!\!\!\!\deg(v)\deg(w)
\LEQ
\half\sum_{v\in V(G)}\!\!\!\!\deg(v)^3.$$
\end{lemma}


\begin{proof}
Let $vw$ be an edge of $G$. Thus $(\deg(v)-\deg(w))^2\geq0$, implying $2\deg(v)\deg(w)\leq\deg(v)^2+\deg(w)^2$. Hence
\begin{align*}
\sum_{vw\in E(G)}\!\!\!\!\deg(v)\deg(w)
\LEQ
\half\sum_{vw\in E(G)}\!\!\!\!\left(\deg(v)^2+\deg(w)^2\right)
&\,=\,
\half\sum_{v\in V(G)}\!\!\!\!\deg(v)^3.\qedhere
\end{align*}
\end{proof}

Again note that
\begin{equation}
\label{SumDegreeCubedDelta}
\sum_{v\in V(G)}\deg(v)^3
\LEQ
\Delta(G)^2\sum_{v\in V(G)}\deg(v)
\,=\,
2\Delta(G)^2\cdot \E{G}.
\end{equation}


%
%
%

\section{Drawings Based on Planar Decompositions}
\label{k5k33}


Let $G$ and $D$ be graphs, such that each vertex of $D$ is a set of vertices of $G$ (called a \emph{bag}).
 For each vertex $v$ of $G$, let $D(v)$ be the subgraph of $D$ induced by the bags that contain $v$. Then $D$ is a \emph{decomposition} of $G$ if:
\begin{itemize}
\item $D(v)$ is connected and nonempty for each vertex $v$ of $G$, and\\[-4ex]
\item $D(v)$ and $D(w)$ touch\footnote{Let $A$ and $B$ be subgraphs of a graph $G$. Then $A$ and $B$ \emph{intersect} if $V(A)\cap V(B)\ne\emptyset$, and $A$ and $B$ \emph{touch} if they intersect or $v\in V(A)$ and $w\in V(B)$ for some edge $vw$ of $G$. } for each edge $vw$ of $G$.
\end{itemize}

Decompositions, when $D$ is a tree, were independently introduced by \citet{Halin76} and \citet{RS-II}. \citet{DK05} 
 first generalised the definition for arbitrary graphs $D$.

Let $D$ be a decomposition of a graph $G$. The \emph{width} of $D$ is the maximum cardinality of a bag. Let $v$ be a vertex of $G$. The number of bags in $D$ that contain $v$ is the \emph{spread} of $v$ in $D$. The \emph{spread} of $D$ is the maximum spread of a vertex of $G$. A decomposition $D$ of $G$ is a \emph{partition} if every vertex of $G$ has spread $1$. 
 The \emph{order} of $D$ is the number of bags. $D$ has \emph{linear order} if $\V{D}\leq c\,\V{G}$ for some constant $c$. If the graph $D$ is a tree, then the decomposition $D$ is a \emph{tree decomposition}. If the graph $D$ is a path, then the decomposition $D$ is a \emph{path decomposition}. The decomposition $D$ is \emph{planar} if the graph $D$ is planar.

A decomposition $D$ of a graph $G$ is \emph{strong} if $D(v)$ and $D(w)$ intersect for each edge $vw$ of $G$. The \emph{treewidth} (\emph{pathwidth}) of $G$, is $1$ less than the minimum width of a strong tree (path) decomposition of $G$.
For each constant $k$, the graphs of treewidth at most $k$ form a proper minor-closed class. Treewidth is particularly important in structural and algorithmic graph theory; see the surveys \cite{Bodlaender-TCS98, Reed03}. 


\citet{WT07} showed that planar decompositions were closely related to crossing number. The next result improves on a bound 
of $(p-1)\,\Delta(G)\,\E{G}$ in \cite{WT07}.

\begin{lemma}
\label{RectDrawing}
Every graph $G$ with a planar partition $H$ of width $p$ has a rectilinear drawing in which each edge crosses at most $2\,\Delta(G)\,(p-1)$ other edges. The total number of crossings,
\begin{equation*}
\RCR{G} \LEQ (p-1) \sum_{v\in V(G)}\deg(v)^2.
\end{equation*}
\end{lemma}

\begin{proof}
The following drawing algorithm is in \cite{WT07}. By the F{\'a}ry-Wagner Theorem, $H$ has a rectilinear drawing with no crossings. Let $\epsilon>0$. Let $D_\epsilon(B)$ be the disc of radius $\epsilon$ centred at each bag $B$ of $H$. For each edge $BC$ of $H$, let $D_\epsilon(BC)$ be the union of all line-segments with one endpoint in $D_\epsilon(B)$ and one endpoint in $D_\epsilon(C)$. For some  $\epsilon>0$, we have $D_\epsilon(B)\cap D_\epsilon(C)=\emptyset$ for all distinct bags $B$ and $C$ of $H$, and $D_\epsilon(BC)\cap D_\epsilon(PQ)=\emptyset$ for all edges $BC$ and $PQ$ of $H$ that have no endpoint in common. For each vertex $v$ of $G$ in bag $B$ of $H$, position $v$ inside $D_\epsilon(B)$ so that $V(G)$ is in general position (no three collinear). Draw every edge of $G$ straight. Thus no edge passes through a vertex.Suppose that two edges $e$ and $f$ cross. Then $e$ and $f$ have distinct endpoints in a common bag, as otherwise two edges in $H$ would cross. (The analysis that follows is new.)\ Say $v_i$ is an endpoint of $e$ and $v_j$ is an endpoint of $f$, where $\{v_1,\dots,v_p\}$ is some bag with $\deg(v_1)\leq\cdots\leq\deg(v_p)$. Without loss of generality, $i<j$. Charge the crossing to $v_j$. The number of crossings charged to $v_j$ is at most 
$$\sum_{i<j}\deg(v_i)\cdot\deg(v_j)\leq (p-1)\deg(v_j)^2$$
So the total number of crossings is as claimed.
\end{proof}


\citet{WT07} proved that every $K_{3,3}$-minor-free graph has a planar partition of width $2$. Thus \cref{RectDrawing} implies the following theorem.

\begin{theorem}
\label{K33Drawing}
Every graph $G$ with no $K_{3,3}$-minor has rectilinear crossing number
\begin{equation*}
\RCR{G}\;\leq\;\sum_{v\in V(G)}\deg(v)^2.
\end{equation*}
\end{theorem}


We now extend \cref{RectDrawing} from planar partitions to planar decompositions. 


\begin{lemma}
\label{Decomp2Drawing}
Suppose that $D$ is a planar decomposition of a graph $G$ with width $p$, in which each vertex $v$ of $G$ has spread at most $s(v)$. Then $G$ has crossing number
\begin{equation*}
\CR{G}\;\leq\;4p\sum_{v\in V(G)}s(v)\cdot\deg(v)^2\enspace.
\end{equation*}
Moreover, $G$ has a drawing with the claimed number of crossings, in which each edge $vw$ is represented by a polyline with at most $s(v)+s(w)-2$ bends.
\end{lemma}


\begin{proof}
For each vertex $v$ of $G$, let $X(v)$ be a bag of $D$ that contains $v$. Think of $X(v)$ as the `home' bag of $v$.  For each edge $vw$ of $G$, let $P(vw)$ be a minimum length path in $D$ between $X(v)$ and $X(w)$, such that $v$ or $w$ is in every bag in $P(vw)$. Let $G'$ be the subdivision of $G$ obtained by subdividing each edge $vw$ of $G$ once for each internal bag in $P(vw)$. Then $D$ defines a planar partition $D'$ of $G'$, where
each original vertex $v$ is in $X(v)$, and each division vertex is in the corresponding bag. We say a division vertex $x$ of $vw$ \emph{belongs} to $v$ and $v$ \emph{owns} $x$, if $x$ corresponds to a bag in $D$ that contains $v$. If $x$ corresponds to a bag that contains both $v$ and $w$, then arbitrarily choose $v$ or $w$ to be the owner of $x$. 

Apply the drawing algorithm in \cref{RectDrawing} to the planar partition $D'$ of $G'$. We obtain a rectilinear drawing of $G'$, which defines a drawing of $G$ since $G'$ is a subdivision of $G$. Each edge $vw$ of $G$ is represented by a polyline with $\max\{|P(vw)|-2,0\}$ bends, which is at most $s(v)+s(w)-2$. We now bound the number of crossings in the drawing of $G'$, which in turn bounds the number of crossings in the drawing of $G$.

Let $\preceq$ be a total order on $V(G)$ such that if $\deg(v)<\deg(w)$ then $v\prec w$ for all $v,w\in V(G)$.

Say edges $e$ and $f$ of $G'$ cross. As proved in \cref{RectDrawing}, $e$ and $f$ have distinct endpoints in a common bag $B'$. Let $x$ and $y$ be these endpoints of $e$ and $f$ respectively. Let $v$ and $w$ be the vertices of $G$ that own $x$ and $y$ respectively. Without loss of generality, $v\preceq w$. Charge the crossing to the pair $(w,B)$, where $B$ is the bag in $D$ corresponding to $B'$. 

Consider a bag $B=\{v_1,\dots,v_p\}$ in $D$, where $v_1\prec\dots\prec v_p$. Thus $\deg(v_1)\leq\cdots\leq\deg(v_p)$. Consider a vertex $v_i\in B$.  If $X(v_i)=B$ then $\deg(v_i)$ edges of $G'$ are incident to $v_i$, which is the only vertex in $B'$ that belongs to $v_i$. If $X(v_i)\neq B$ then there are at most $\deg(v_i)$ division vertices in $B'$ that belong to $v_i$, and there are at most $2\deg(v_i)$ edges of $G'$ incident to a division vertex in $B'$ that belongs to $v_i$ (since each division vertex has degree $2$ in $G'$). Thus the number of crossings charged to $(v_i,B)$ is at most $$\sum_{j=1}^i2\deg(v_j)\cdot 2\deg(v_i)\leq 4i\,\deg(v_i)^2\leq 4p\deg(v_i)^2.$$

For each vertex $v$ of $G$, since $v$ is in at most $s(v)$ bags of $D$, the number of crossings charged to some pair $(v,B)$ is at most $4p\cdot s(v)\cdot \deg(v)^2$. Hence the total number of crossings is at most 
\begin{equation*}
4p\sum_{v\in V(G)}s(v)\cdot \deg(v)^2.\qedhere
\end{equation*}
\end{proof}


Our first application of \cref{Decomp2Drawing} concerns graphs of bounded treewidth.  \citet{DO95} proved that every graph with bounded treewidth has a tree decomposition of bounded width, in which the spread of each vertex is bounded by its degree. In particular:

\begin{theorem}[\citet{DO95}]
\label{Treewidth2Decomp}
Every graph $G$ with treewidth $k$ has a tree decomposition of width $2^{k+1}(k+1)-1$ and order $\V{G}$, in which the spread of each vertex $v$ is at most $3^{2^k}\deg(v)+1$. 
\end{theorem}


\cref{Decomp2Drawing,Treewidth2Decomp} imply:

\begin{theorem}
Every graph $G$ with treewidth $k$ has crossing number
\begin{align*}
\CR{G}
\LEQ 
(2^{k+1}(k+1)-2)\sum_{v\in V(G)}(3^{2^k}\deg(v)+1)\cdot\deg(v)^2
\LEQ c_k\sum_{v\in V(G)}\deg(v)^3.
\end{align*}
\end{theorem}

\cref{Decomp2Drawing} also leads to the following upper bound on the crossing number. 

\begin{lemma}
\label{NewGeneralBound}
Let $D$ be a planar decomposition of a graph $G$, such that every bag in $D$ is a clique in $G$, and every pair of adjacent vertices in $G$ are in at most $c$ common bags in $D$. Then $$\CR{G}\leq c\sum_{vw\in E(G)}\!\!\!\!\deg(v)\deg(w).$$
\end{lemma}

\begin{proof}
Draw $G$ as in the proof of \cref{Decomp2Drawing}. We now count the crossings in $G$ between edges $vw$ and $xy$ that have no common endpoint. Each crossing between $vw$ and $xy$ can be charged to a bag $B$ that contains distinct vertices $p$ and $q$, where $p\in\{v,w\}$ and $q\in\{x,y\}$. Since $B$ is a clique, $pq$ is an edge of $G$. Charge the crossing to the pair $(pq,B)$. At most one crossing between $vw$ and $xy$ is charged to $(pq,B)$. Thus at most $\deg(p)\deg(q)$ crossings are charged to $(pq,B)$. Since $p$ and $q$ are in at most $c$ common bags, the number of crossings charged to $pq$ is at most $c\deg(p)\deg(q)$. Thus the total number of crossings between edges with no common endpoint is at most $c\sum_{pq\in E(G)}\deg(p)\deg(q)$. The result follows since it is folklore that \CR{G} equals the minimum, taken over all drawings of $G$, of the number of crossings between pairs of edges of $G$ with no endpoint in common; see \citep{FPSS12} for example. 
\end{proof}

\section{Interval Graphs and Chordal Graphs}
\label{pw}

A graph is \emph{chordal} if every induced cycle is a triangle. An \emph{interval graph} is the intersection graph of a set of intervals in $\mathbb{R}$. Every interval graph is chordal.

\begin{theorem}
\label{ConvexIntervalGraph}
Every interval graph $G$ has convex crossing number
\begin{align*}
\CCR{G}\;\leq&\; \half(\omega(G)-2)\sum_{v\in V(G)}\deg(v)(\deg(v)-1)\\
		 \leq&\; (\omega(G)-2)(\omega(G)-1)(\Delta(G)-1)\V{G}.
\end{align*}
\end{theorem}

\begin{proof}
\citet{JL84} proved that $G$ is an interval graph if and only if there is a linear order $\preceq$ of $V(G)$ such that if $u\prec v\prec w$ and $uw\in E(G)$ then $uv\in E(G)$.
Orient the edges of $G$ left to right in $\preceq$.
Position $V(G)$ on a circle in the order of $\preceq$, with the edges drawn straight.
Say edges $xy$ and $vw$ cross.
Without loss of generality, $x\prec v\prec y\prec w$.
Thus $vy\in E(G)$.
Charge the crossing to $vy$.
Say the out-neighbours of $v$ are $w_1,\dots,w_d$.
The in-neighbourhood of each $w_i$ is a clique including $v$.
Hence each $w_i$ has at most $\omega(G)-2$ in-neighbours to the left of $v$.
Now $v$ has $d-i$ neighbours to the right of $w_i$.
Thus the number of crossings charged to $vw_i$ is at most $(\omega(G)-2)(d-i)$.
Hence the number of crossings charged to outgoing edges at $v$ is at most $\half(\omega(G)-2)(d-1)d$.
Therefore the total number of crossings is at most $\half\sum_v(\omega(G)-2)(d_v-1)d_v$,
where $d_v$ is the out-degree of $v$.
The other claims follow since $\E{G}<(\omega(G)-1)\V{G}$.
\end{proof}

It is well known that the pathwidth of a graph $G$ equals the minimum $k$ such that $G$ is a spanning subgraph of an interval graph $G'$ with $\omega(G')\leq k+1$.

\begin{theorem}\label{pw-cx}
Every graph $G$ with pathwidth $k$ has convex crossing number $$\CCR{G}\leq k^2\cdot\Delta(G)\cdot\V{G}.$$
\end{theorem}

\begin{proof}
$G$ is a spanning subgraph of an interval graph $G'$  with $\omega(G')\leq k+1$.
Apply the drawing algorithm in the proof of \cref{ConvexIntervalGraph} to $G'$.
Say edges $xy$ and $vw$ of $G$ cross.
Without loss of generality, $x\prec v\prec y\prec w$.
Thus $vy\in E(G')$.
Charge the crossing to $vy$.
Now $v$ has at most $\Delta(G)$ neighbours in $G$ to the right of $y$.
The in-neighbourhood of $y$ is a clique in $G'$ including $v$.
Hence $y$ has at most $k$ neighbours to the left of $v$.
Thus the number of crossings charged to $vy$ is at most $k\cdot\Delta(G)$.
Since $G'$ has less than $k\cdot\V{G}$ edges,
the total number of crossings is at most $k^2\cdot\Delta(G)\cdot\V{G}$.
\end{proof}

\begin{lemma}
\label{OuterDecomp2ConvexDrawing}
Let $D$ be an outerplanar decomposition of a graph $G$. Then $G$ has a convex drawing such that if two edges $e$ and $f$ cross, then some bag of $D$ contains both an endpoint of $e$ and an endpoint of $f$.
\end{lemma}

\begin{proof}
Assign each vertex $v$ of $G$ to a bag $B(v)$ that contains $v$.
Fix a crossing-free convex drawing of $D$.
Replace each bag $B$ of $D$ by the set of vertices of $G$ assigned to $B$.
Draw the edges of $G$ straight.
Consider two edges $vw$ and $xy$ of $G$.
Thus there is a path $P$ in $D$ between $B(v)$ and $B(w)$ and every bag in $P$ contains $v$ or $w$.
Similarly, there is a path $Q$ in $D$ between $B(x)$ and $B(y)$ and every bag in $Q$ contains $x$ or $y$.
Now suppose that $vw$ and $xy$ cross.
Without loss of generality, the endpoints are in the cyclic order $(v,x,w,y)$.
Thus in the crossing-free convex drawing of $D$, the vertices $(B(v),B(x),B(w),B(y))$ appear in this cyclic order.
Since $D$ is crossing-free, $P$ and $Q$ have a bag $X$ of $D$ in common.
Thus $X$ contains $v$ or $w$, and $x$ or $y$.
\end{proof}

\begin{theorem}
\label{ChordalDrawing}
Every chordal graph $G$ has convex crossing number
 $$\CCR{G}\leq\sum_{vw\in E(G)}\deg(v)\deg(w).$$
\end{theorem}

\begin{proof}
It is well known that every chordal graph has a strong tree decomposition in which each bag is a clique \citep{Diestel4}. Since trees are outerplanar, by \cref{OuterDecomp2ConvexDrawing}, $G$ has a convex drawing such that if two edges $vw$ and $xy$ of $G$ cross, then some bag $B$ of $D$ contains $v$ or $w$, and $x$ or $y$. Say $B$ contains $v$ and $x$. Since $B$ is a clique, $vx$ is an edge. Charge the crossing to $vx$. In every crossing charged to $vx$, one edge is incident to $v$ and the other edge is incident to $x$. Since edges are drawn straight, no two edges cross twice. Thus the number of crossings charged to $vx$ is at most $\deg(v)\deg(x)$. Hence the total number of crossings is as claimed.
\end{proof}

Note that the upper bound in \cref{ChordalDrawing} is within a constant factor of being tight for sufficiently dense chordal graphs for which all the vertex degrees are within a constant factor of each other. In particular, consider a chordal graph $G$ in which every vertex degree is within a constant factor of $k$. Then the upper bound in \cref{ChordalDrawing} is \Oh{k^3n}, and the crossing lemma gives a lower bound of $\Omega(m^3/n^2)=\Omega(k^3n)$ (assuming $m\geq(3+\epsilon)n$). 

Also note that there are chordal graphs for which the lower bound $\Omega(m^3/n^2)$ is far from the crossing number. Let $G$ be the split graph obtained from a triangle $T$ by adding a set of $n$ independent vertices, each  adjacent to every vertex in $T$. Then $\Omega(m^3/n^2)=\Omega(n)$ but $\CR{G}=\Theta(n^2)$ since $G$ contains $K_{3,n}$ as a subgraph.


\begin{theorem}
\label{ktreesDrawing}
Every chordal graph $G$ with no $(k+2)$-clique has convex crossing number $$\CCR{G}\leq 16k^2\cdot\Delta(G)\cdot\V{G}.$$
\end{theorem}

\begin{proof}
It is well known that $G$ has less than $kn$ edges. Thus the claim follows from \cref{arboricity-trick} and \cref{ChordalDrawing}.
\end{proof}

\section{Excluding a Fixed Minor}
\label{minors-cross}

In this section we prove our main result (\cref{main}): for every graph $H$ there is a constant $c=c(H)$, such that every $H$-minor-free graph $G$ has a crossing number at most $c\,\Delta(G)\cdot \V{G}$. The proof is based on Robertson and Seymour's rough characterisation of $H$-minor-free graphs, which we now introduce. For an integer $h\geq 1$ and a surface $\Sigma$, \citet{RS-XVI} defined a graph $G$ to be $h$-\emph{almost embeddable} in $\Sigma$ if $G$ has a set $X$ of at most $h$ vertices (called \emph{apices}) such that $G\graphminus X$ can be written as $G_0\cup G_1 \cup \dots \cup G_h$ such that:\\[-4ex]
\begin{itemize}
\item $G_0$ has an embedding in $\Sigma$.
\item The graphs $G_1 , \dots , G_h$ (called \emph{vortices}) are pairwise disjoint.
\item There are faces\footnote{Recall that we identify a face with the set of vertices on its boundary.} $F_1 , \dots , F_h$ of the embedding of $G_0$ in $S$, such that $F_i=V(G_0) \cap V(G_i)$ for $i\in\{1,\dots,h\}$.
\item For $i\in\{1,\dots,h\}$, if $F_i=(u_{i,1}, u_{i,2},\dots,u_{i,|F_i|})$ in clockwise order about the face, then $G_i$ has a strong $|F_i|$-path decomposition $Q_i$ of width at most $h$, such that each vertex $u_{i,j}$ is in the $j$-th bag of $Q_i$.
\end{itemize}


%
%

\begin{theorem}\label{almost-cross}
For all integers $h\geq 1$ and $\gamma\geq 0$, there is a constant $k=k(h,\gamma)\geq h$, such that every graph $G$ that is $h$-almost embeddable in some surface $\Sigma$ with Euler genus at most $\gamma$, has crossing number at most $k\,\Delta(G) \cdot \V{G}$.
\end{theorem}

\begin{proof}
Let $X$ and $\{G_0, G_1, \dots, G_h\}$ be the parts of $G$ as specified in the definition of $h$-almost embeddable graph. 
Let $\Delta:=\Delta(G)$ and $n:=\V{G}$. 
Start with an embedding of $G_0$ in $\Sigma$. 
For each $i\in \{1,\dots, h\}$, draw vortex $G_i$ inside of the face $F_i$ on $\Sigma$, as prescribed in \cref{pw-cx}. 
Then the resulting drawing of $G\graphminus X$ in $\Sigma$ has at most $h^2\Delta n$ crossings. 
Replace each crossing by a dummy degree-$4$ vertex. 
The resulting graph $G'$ is embedded in $\Sigma$
By \cref{BoundedGenus}, for some constant $c=c(\Sigma)$, 
$$\CR{G'}\leq 
c\,\sum_{v\in V(G')}\deg_{G'}(v)^2
\LEQ c\,\sum_{v\in V(G)}\deg(v)^2 + c4^2h^2\Delta n.$$ 
Since $\CR{G\graphminus X}\leq h^2\Delta n + \CR{G'}$, we conclude that  
$$\CR{G\graphminus X}
\LEQ c\,\sum_{v\in V(G)}\deg(v)^2+(16c+1)h^2\Delta n
\LEQ 2c\Delta\, \E{G} + (16c+1)h^2\Delta n. $$
By \cref{AddVertex} below, 
\begin{align*}
\CR{G}  \LEQ \CR{G-X} + h \Delta \,\E{G} 
 \LEQ (2c+h)\Delta \E{G}+(16c+1)h^2\Delta n.
\end{align*}
In the $H$-minor-free graph $G$, the number of edges is at most $c'\,\V{G}$, for some constant $c'=c'(H)$ (see \citep{Mader68}). 
Thus $\CR{G}\leq \big( (16c+1)h^2 + c' (2c+h) \big) \Delta \V{G}$. 
\end{proof}

\begin{lemma}
\label{AddVertex}
For every vertex $v$ of a graph $G$, 
$$\CR{G} \leq \CR{G-v} + \deg_G(v) \, \E{G}.$$
\end{lemma}

\begin{proof}
\citet{PachToth-Geom00} proved that for every edge $e$ of a graph $G$, 
$$\CR{G} \leq \CR{G-e} + \E{G-e}.$$
The claim follows by applying this result to each edge incident to $v$. 
\end{proof}


Let $G_1$ and $G_2$ be disjoint graphs. Suppose that $C_1$ and $C_2$ are cliques of $G_1$ and $G_2$ respectively, each of size $k$, for some integer $k\geq 0$. Let $C_1=\{v_1,v_2,\dots,v_k\}$ and $C_2=\{w_1,w_2,\dots,w_k\}$. Let $G$ be a graph obtained from $G_1\cup G_2$ by identifying $v_i$ and $w_i$ for each $i\in \{1,\dots,k\}$, and possibly deleting some of the edges $v_iv_j$. Then $G$ is a \emph{$k$-clique-sum} of $G_1$ and $G_2$ \emph{joined} at $C_1=C_2$. An $\ell$-clique-sum for some $\ell\leq k$ is called a $(\leq k)$-clique-sum.

The following rough structural characterisation of $H$-minor-free graphs is a deep theorem by \citet{RS-XVI}; see the surveys  \cite{KM-GC07,NorinSurvey15}.

\begin{theorem}\textup{(Graph Minor Structure Theorem \cite{RS-XVI})}\label{RS-minors}
For every graph $H$, there is a positive integer $h=h(H)$, such that every $H$-minor-free graph $G$ can be obtained by $(\leq h)$-clique-sums of graphs that are $h$-almost embeddable in some surface in which $H$ cannot be embedded.
\end{theorem}

Our main result, \cref{main}, is directly implied by \cref{RS-minors} and the following theorem.

\begin{theorem}\label{clique-sum}
For all integers $h\geq 1$ and $\gamma\geq 0$ there is a constant $c=c(h,\gamma)\geq h$, such that every graph $G$ that can be obtained by $(\leq h)$-clique-sums of graphs that are $h$-almost embeddable in some surface of Euler genus at most $\gamma$ has crossing number at most $c\,\Delta(G) \cdot \V{G}$.
\end{theorem}


The remainder of this section is dedicated to proving \cref{clique-sum}. Let $\Delta:=\Delta(G)$. Let $U$ be the set of integers $\{1,2,\dots,|U|\}$, such that $\{\gi\,:\, i\in U\}$ is the set (of minimum cardinality) of graphs such that for all $i\in U$, \gi\ is  $h$-almost embeddable in some surface of Euler genus at most $\gamma$, and $G$ is obtained by $(\leq h)$-clique-sums of graphs in the set. These graphs can be ordered $\g{1}, \dots, \g{|U|}$, such that for all $j \geq 2$, there is a minimum integer $i<j$, such that \gi\ and \gj\ are joined at some clique $C$ in the construction of $G$. We say \gj\ is a \emph{child} of \gi, \gi\ is a \emph{parent} of \gj, and $P_j:=V(C)$ is the \emph{parent clique} of \gj. We consider the parent clique of $G_1$ to be the empty set; that is, $P_1=\emptyset$. This defines a rooted tree $T$ with vertex set $U$ where $ij$ is an edge of $T$ if and only if \gj\ is a child of \gi. Let $U_i$ denote the set of children of $i$ in $T$. Let $T_i$ denote the subtree of $T$ rooted at $i$. For $S\subseteq V(T)$, let $G[S]$ be the graph induced in $G$ by $\bigcup \{V(\g{\ell})\,:\, \ell\in S\}$. For example, for $S=\{i\}$, then $G[S]$ is a spanning subgraph of $\gi$.

The proof outline is as follows. For each \gi, $i\in U$, we define an auxiliary graph $\KK_i$ (closely related to \gi),
such that $$\E{\KK_i}= \Oh{\sum_{v\in V(\gi)\setminus P_i} \deg_G(v)}.$$ We draw each $\KK_i$ in the plane with at most $f(h)\Delta \E{\KK_i}$ crossings, for some function $f$ of $h$. We then join the drawings of $\KK_1, \dots,\KK_{|U|}$ into a drawing of $G$, where the price of the joining is at most an additional $f(h)\Delta$ crossings for each edge of $\KK_i$, $i\in U$. Thus the crossing number of $G$ is at most $f_1(h)\Delta\sum_{i\in U} \E{\KK_i}$, which, by the above claim on the number of edges of $\KK_i$, is at most
\begin{align*}
f_2(h)\,\Delta\,\sum_{i\in U}\sum_{v\in V(\gi)\setminus P_i}\!\!\!\!\!\deg_G(v)
\;\leq\; f_2(h)\,\Delta\,\sum_{v\in V(G)}\!\!\!\deg_G(v) 
\;=\;2f_2(h)\,\Delta\, \E{G}
\;\leq \;f_3(h)\,\Delta\, \V{G},
\end{align*} 
which is the desired result.



\medskip
\noindent{\bf Defining {\boldmath $\KK_i$}.} For each $i\in U$, let $\gip:=G_i\graphminus P_i$.  Note that, for each $v\in V(G)$, there is precisely one value $t\in U$ for which $v\in\gp{t}$. Thus $\{V(\gp{1}),\dots, V(\gp{|U|})\}$ is a partition of $V(G)$. For each $i\in U$, define $\KK_i$ as follows. Start with \gip. For each child \gj\ of \gi\ (that is, for each $j\in U_i$), add a new vertex $c_j$ to \gip. For each edge $vw\in E(G)$ such that $v\in V(\gip) \cap P_j$ (that is, $v\in P_j\setminus P_i$) and $w\in \gp{\ell}$ where $\ell\in V(T_j)$,  connect $v$ and $c_j$ by an edge. Subdivide that edge once and label the subdivision vertex by the triple $(v,w, \PPP_{vw})$, where $\PPP_{vw}$ is the path in $T$ from $i$ to $\ell$ (thus, $\PPP_{vw}=(i,j,\dots, \ell)$). 
The resulting graph is $\KK_i$. Note that for each $v$ in $\gip$, $\deg_{\KK_i}(v)=\deg_{G\graphminus P_i} (v)$.

\medskip
\noindent{\bf Drawing {\boldmath $\KK_i$}.}
Suppose that for each $i\in U$, we remove each $c_j$, $j\in U_i$, from $\KK_i$. Consider the union of the resulting graphs, over all $i\in U$. Suppose that, for each vertex labelled $(v,w,\PPP_{vw})$ in the union, we connect this vertex and $w$ by an edge. The resulting graph is a subdivision of $G$. This is the strategy that we will follow when constructing a drawing of $G$. Namely, first draw each $K_i$, and then take the (disjoint) union of all the drawings. Next, remove all $c_j$'s. Finally, to obtain a drawing of $G$, route each missing edge of $G$. In particular, for a missing edge between $(v,w,\PPP_{vw})$ and $w$ with $\PPP_{vw}=(i,j,\dots, \ell)$, we route that edge from $(v,w,\PPP_{vw})$ in the drawing of $K_i$, through the drawing of $K_j$, etc., until we finally reach $w$ in the drawing of $K_\ell$.

We first claim that the number of edges in $\KK_i$ is as stated in the outline. In addition to the edges in $E(\gip)$, $\KK_i$ contains two edges for each edge $vw\in E(G)$, such that $v\in \gip$ and $w\in \gp{\ell}$, where $\ell \in V(T_i)\setminus i$. Thus  $$\E{\KK_i}\leq 2\sum_{v\in V(\gip)} \deg_G(v)=2\sum_{v\in V(\gi)\setminus P_i} \deg_G(v).$$




\begin{lemma}\label{drawing-k}
For each $i\in U$, the crossing number of $\KK_i$ is at most $f(h)\Delta \E{\KK_i}$.
\end{lemma}

\begin{proof}
For each \gi, let $A_i$ denote the set of apex vertices of \gi\ that are not in $P_i$. Remove all the vertices of $A_i$ from $\KK_i$. We now prove that the resulting graph $\KK_i\graphminus A_i$ can be drawn in some surface $\Sigma$ of Euler genus at most $\gamma$ with at most $f(h)\Delta \E{\KK_i\graphminus A_i}$ crossings. That will complete the proof since \cref{BoundedGenus} implies that $\CR{\KK_i\graphminus A_i}\leq f(h)\Delta \E{\KK_i\graphminus A_i}$, the same way it did in the proof of \cref{almost-cross}.  Then we add back each vertex of $A_i$ to the drawing of $\KK_i\graphminus A_i$ at some arbitrary position in the plane and draw its incident edges to obtain a drawing of $\KK_i$. As in the proof of \cref{almost-cross}, $\CR{\KK_i}\leq \CR{\KK_i\graphminus A_i} + h\Delta \E{\KK_i}\leq f_2(h)\Delta \E{\KK_i}$.



Thus it remains to prove that $\KK_i\graphminus A_i$ can be drawn in $\Sigma$ with at most $f(h)\Delta \E{\KK_i\graphminus A_i}$ crossings. The graph $Q:=G_i^-\graphminus A_i$ is an apex-free $h$-almost embeddable graph on $\Sigma$, with parts $\{Q_0,Q_1, \dots, Q_h\}$, where $Q_0$ is the  subgraph of $Q$ embedded in $\Sigma$ and $\{Q_1, \dots, Q_h\}$ are its vortices. For each $j\in U_i$, let $C_j$ denote the subgraph of $\KK_i\graphminus A_i$ induced by $c_j$ and the vertices at distance at most two from $c_j$. The vertices at distance $2$ from $c_j$ form a clique $C\subseteq(P_j\setminus P_i)\setminus A_i\subseteq K_i\graphminus A_i$.  It is simple to verify that $C_j$ has a strong tree decomposition $J$ of width at most $h+2$, where $J$ is a rooted star whose root bag contains $C\cup\{c_j\}$; for each $(v,w,\PPP_{vw})\in C_j$ (where $v\in C$), $J$ contains a leaf bag with $\{w, c_j, (v,w,\PPP_{vw})\}$; if $v\notin C$, then $v$ is in $A_i$ and the leaf bag contains $\{c_j, (v,w,\PPP_{vw})\}$.

We now add the vortices and $C_j$'s to $Q_0$ to obtain a drawing of $\KK_i\graphminus A_i$ in $\Sigma$ while creating at most $f(h)\Delta \E{\KK_i\graphminus A_i}$ crossings in $\Sigma$.

For each $j\in U_i$, $C_j$ is joined to a clique $C$ of $Q$. If $C$ contains a vertex $v$ of a vortex $Q_\ell$, where $\ell\in\{1,\dots, h\}$, then each vertex of $C$ is in $Q_\ell$. In that case, we say that $C_j$ \emph{belongs to} the face $F_\ell$ of the embedding of $Q_0$ in $\Sigma$.  Otherwise, all the vertices of $C$ are in $Q_0$. In that case, an extended version of the graph minor decomposition theorem (see \cite{KM-GC07}) states that $|C|\leq 3$ and moreover, if $|C|=3$, then the $3$-cycle induced by $C$ is a face in $Q_0$. In that case, we say that $C_j$ \emph{belongs to} that face. If $|C|\leq 2$ we assign $C_j$ to any face of $Q_0$ incident to all the vertices of $C$.

Now consider a face $F$ of $Q_0$. If $F=F_\ell$ for some $\ell$ ($1\le \ell\le h$), take its vortex $Q_\ell$, and all $C_j$, $j\in U_i$, that belong to $F$. Let $F'$ be the subgraph of $\KK_i\graphminus A_i$ induced by the union of $F$ and all of these. If $F$ is not one of the vortex faces, then we define $F'$ similarly by taking the union of $F$ and all $C_j$, $j\in U_i$, that belong to $F$.
If $F'$ contains a vortex $Q_\ell$, consider a strong path decomposition $P_F$ of $F\cup Q_\ell$, as defined by the $h$-almost embedding. If $F$ has no vortex, then its strong path decomposition $P_F$ is just a bag containing $|F|\leq 3$ vertices of $F$ in it. For each $C_j$ in $F'$, the join clique $C$ of $C_j$ is in some bag of $P_F$. Extend the decomposition $P_F$ and $J$ by adding an edge between that bag of $P_F$ and the root of $J$. It is simple to verify that the resulting strong tree decomposition of $F'$ can be converted into a strong path decomposition of width at most $h+3$. Thus by \cref{pw-cx}, $F'$ can be drawn inside of $F$ with at most $(h+3)^2\Delta \V{F'}$ crossings. Accounting for all the faces of $Q_0$ gives $f_4(h)\Delta \E{\KK_i\graphminus A_i}$ bound on the number of crossings in the resulting drawing of $\KK_i\graphminus A_i$ in $\Sigma$, as required.
\end{proof}

In addition to having at most as many crossings as proved in \cref{drawing-k}, we need a drawing of $\KK_i$ with  the following additional properties.

\begin{lemma}\label{drawing-k-extra}
For each $i\in U$, there is a drawing of $\KK_i$ with at most $f(h)\Delta \E{\KK_i}$ crossings such that:\\[-4ex]

\begin{enumerate}
\item[(1)] No pair of vertices in $\KK_i$ has the same $x$-coordinate.\\[-4ex]
\item[(2)] For each $j\in U_i$, there is a square\footnote{By a \emph{square}, we mean a 4-sided regular polygon together with its interior.} $D_j$ such that $D_j\cap \KK_i=c_j$, and $c_j$ is an internal point of the top side of $D_j$, and no vertex in $V(\KK_i)\setminus \{c_j\}$ has the same $x$-coordinate as any point of $D_j$.\\[-4ex]
\item[(3)] For any two $j, t\in U_i$, there is no line parallel to the $y$-axis that intersects both $D_j$ and $D_t$.\\[-4ex]
\item[(4)] Moreover, given a circular ordering $\sigma_j$ of the edges incident to each vertex $c_j$ in $\KK_i$, $j\in U_i$, there is a drawing of $\KK_i$ that satisfies $(1)$--$(3)$ such that the circular ordering of the edges incident to each $c_j$ respects $\sigma_j$.
\end{enumerate}
\end{lemma}

\begin{proof}
Apply \cref{drawing-k} to $\KK_i$ to obtain a drawing of $\KK_i$ with at most $s:=f(h)\Delta \E{\KK_i}$ crossings. Clearly, the edges incident to $c_j$ can be bent without changing the number of crossings such that there is a small enough square $D_j$ that satisfies all the properties imposed on $D_j$, as stated in (2). Similarly, condition (3) is satisfied by shrinking the squares further, if necessary. By an appropriate rotation, the conditions on the $x$- and $y$-coordinates imposed in (1)--(3) are satisfied. 

Consider a disk $C_j$ centered at $c_j$, such that $c_j$ is the only vertex of $\KK_i$ that intersects $C_j$, and the only edges of $\KK_i$  that intersect $C_j$ are the edges incident to $c_j$. Order the edges around $c_j$ with respect to $\sigma_j$ by moving (that is, bending) the edges incident to $c_j$ within $C_j\setminus D_j$. This may introduce new crossings. Each new crossing point is in $C_j\setminus D_j$ and thus it occurs between a pair of edges incident to $c_j$. There are at most $h\Delta$ edges incident to $c_j$. Thus each edge incident to $c_j$ gets at most $h\Delta$ new crossings. Therefore, the resulting drawing of $\KK_i$ satisfies conditions $(1)$--$(4)$ and has at most $s+h\Delta \E{\KK_i}\leq f'(h)\Delta \E{\KK_i}$ crossings.
\end{proof}

\noindent{\bf Joining the {\boldmath $\KK_i$'s} into a drawing of {\boldmath $G$}.}
We obtain a drawing of $G$ from the union of the drawings of $\KK_i$, $i\in U$, as follows.
Join the drawings of these graphs in the order determined by a breath-first search on $T$, as follows. For each $G_i$, consider a drawing of $\KK_i$ together with the squares incident to its children, as defined in \cref{drawing-k-extra}. For each $j\in U_i$, place the drawing of $\KK_j$ strictly inside of the square $D_j$ of $\KK_i$ (while scaling the drawing of $\KK_j$, if necessary). Denote by \KK\ the resulting drawing of ${\bigcup_i \KK_i}$. This procedure introduces no new crossings, thus by \cref{drawing-k-extra}, the number of crossings in $\KK$ is at most $\sum_{i\in U}f'(h)\Delta\E{\KK_i}$.

We still have the freedom to choose an arbitrary ordering $\sigma_j$ (cf. \cref{drawing-k-extra}(4)) to be used in the drawing of $\KK_j$. Define the ordering $\sigma_j$ of edges around each vertex $c_j$ ($j\in U\setminus \{1\}$) as follows. Consider an edge $e_1$ joining $c_j$ and $(v,w,\PPP_{vw})$, and an edge $e_2$ joining $c_j$ and $(a,b,\PPP_{ab})$. Define $e_1\leq_{\sigma_j}e_2$ if the $x$-coordinate of $w$ in \KK\ is less than the $x$-coordinate of $b$ in \KK. If $w=b$, order $e_1$ and $e_2$ according to the $x$-coordinates of $v$ and $a$.
Since no pair of vertices in \KK\ have the same $x$-coordinate, $\sigma_j$ is a linear order of the edges incident to $c_j$.

For each $j\in U\setminus \{1\}$, we may assume that the graph induced in \KK\ by $c_j$ and its neighbours (the subdivision vertices), is a crossing-free star in \KK; that is, no edge of this star is crossed by any other edge of \KK.


For each $i\in U$, remove each $c_j$, $j\in U_i$, from \KK. The subdivision vertices of \KK\ become degree-1 vertices. For each such subdivision vertex $(v,w,\PPP_{vw})$, where $\PPP_{vw}=(i,j,\dots,\ell)$, draw an edge from $(v,w,\PPP_{vw})$ to the point on the top side of the square $D_j$ that has the same $x$-coordinate as the vertex $w$ in $\KK$. Since $w\in G[T_j]\graphminus P_j$, it is drawn inside $D_j$, and thus such a point on the top side of $D_j$ exists. If $w$ is an endpoint of $s\geq 2$ such edges, draw $s$ points very close together on the top side of $D_j$ and connect each of the $s$ edges to one of these $s$ points in the order $\sigma_j$. (In fact, imagine that these points are almost overlapping; that is, their $x$-coordinates are almost the same as that of $w$ in \KK). Since the star incident to $c_j$ is crossing-free in \KK, this can be done so that the resulting drawing $\KK_i^-$ has the same number of crossings as $\KK_i$. Label each point on the top side of $D_j$ by the same label as the subdivision vertex it is adjacent to. (In fact, consider that point on the top side of $D_j$ to be the subdivision vertex instead of the old one). Draw a line-segment between each subdivision vertex $(v,w,\PPP_{vw})$ on the top side of $D_j$ and $w$. Call these segments \emph{vertical segments}. This defines a drawing of $G$. We now prove that the number of crossings in $G$ does not increase much compared to the number of crossings in \KK. Specifically, it increases by at most $f(h)\Delta \sum_{i\in U}\E{\KK_i}$.

Note that \cref{drawing-k-extra} does not define the square $D_1$. Let $D_1$ be the whole plane. For each $i\in U$, let $D_i^-$ be the region  $D_i\setminus \{\bigcup_{j\in U_i}D_j\}$. Denote by $d_i$ the number of crossings in the drawing of $G$ restricted to $D_i^-$. Then $\CR{G}\leq \sum_i d_i $.


We now prove that for each $i\in U$, $d_i\leq f(h)\Delta \E{\KK_i}$, which will complete the proof. Quantity $d_i$ is at most the number of crossings in $\KK_i^-$  plus the number of crossings caused by the vertical segments intersecting $D_i^-$. By construction (cf.\ properties (2) and (3) of \cref{drawing-k-extra}), each vertical segment that intersects $D_i^-$ is a part of an edge that has one endpoint in $\gp{s}$ where $i\in V(T_s)\setminus s$ (that is, $\gp{s}$ is an ancestor of $\gp{i}$) and its other endpoint is either in $G_i^-$ (and, thus in $\KK_i^-$) or is in a descendent $G_\ell^-$ of $G_i^-$. Thus the number of vertical segments that cross $D_i^-$ is at most $f(h)\Delta$. No pair of vertical segments cross in $D_i^-$ due to their ordering. Thus each new crossing in $D_i^-$ (that is, a crossing not present in the drawing of $\KK_i^-$) occurs between a vertical segment and an edge of $\KK_i^-$. Thus each edge of $\KK_i^-$ accounts for at most $f(h)\Delta$ new crossings, and thus $d_i\leq f(h)\Delta \E{\KK_i^-}\leq f(h)\Delta \E{\KK_i}$, as desired. This completes the proof of \cref{clique-sum}.

  \let\oldthebibliography=\thebibliography
  \let\endoldthebibliography=\endthebibliography
  \renewenvironment{thebibliography}[1]{%
    \begin{oldthebibliography}{#1}%
      \setlength{\parskip}{0.25ex}%
      \setlength{\itemsep}{0.25ex}%
  }%
  {%
    \end{oldthebibliography}%
  }


\def\soft#1{\leavevmode\setbox0=\hbox{h}\dimen7=\ht0\advance \dimen7
  by-1ex\relax\if t#1\relax\rlap{\raise.6\dimen7
  \hbox{\kern.3ex\char'47}}#1\relax\else\if T#1\relax
  \rlap{\raise.5\dimen7\hbox{\kern1.3ex\char'47}}#1\relax \else\if
  d#1\relax\rlap{\raise.5\dimen7\hbox{\kern.9ex \char'47}}#1\relax\else\if
  D#1\relax\rlap{\raise.5\dimen7 \hbox{\kern1.4ex\char'47}}#1\relax\else\if
  l#1\relax \rlap{\raise.5\dimen7\hbox{\kern.4ex\char'47}}#1\relax \else\if
  L#1\relax\rlap{\raise.5\dimen7\hbox{\kern.7ex
  \char'47}}#1\relax\else\message{accent \string\soft \space #1 not
  defined!}#1\relax\fi\fi\fi\fi\fi\fi}

\end{document}